\documentclass{article}

\usepackage{amssymb,amsmath,amsthm}
\usepackage{graphicx}
\usepackage{proof}
\usepackage{cite}
\usepackage[UKenglish]{babel}
\usepackage[all]{xy} 
\usepackage{hyperref}
\usepackage{times}

\hypersetup{
   colorlinks,%
   citecolor=blue,%
   filecolor=black,%
   linkcolor=red,%
   urlcolor=black
 }

\newcommand{\lr}[1]{\langle #1 \rangle}
\newcommand{\lra}{\leftrightarrow}

\newcommand{\ML}{\ensuremath{\mathcal{L}(\Box)}}

\newcommand{\LEA}{\mathcal{L}_\bullet}

\newcommand{\ALA}{\mathbb{ALA}}

\newcommand{\BP}{\textbf{P}}

\newcommand{\M}{\ensuremath{\mathcal{M}}}
\newcommand{\N}{\ensuremath{\mathcal{N}}}

\renewcommand{\phi}{\varphi}

\newcommand{\weg}[1]{}

\newtheorem{theorem}{Theorem}
\newtheorem{lemma}[theorem]{Lemma}
\newtheorem{definition}[theorem]{Definition}

\newtheorem{proposition}[theorem]{Proposition}

\newtheorem{corollary}[theorem]{Corollary}
\newtheorem{fact}[theorem]{Fact}
\newtheorem{conjecture}[theorem]{Conjecture}

\begin{document}

\title{Bimodal logics with contingency and accident\thanks{This research is supported by the youth project 17CZX053 of National Social Science Fundation of China.}}
\author{Jie Fan\\
\small School of Philosophy, Beijing Normal University  \\
\small \texttt{fanjie@bnu.edu.cn}}
\date{Submitted on 05 Dec 2017\\
(Any comments or corrections are welcome!)}

\maketitle

\begin{abstract}
Contingency and accident are two important notions in philosophy and philosophical logic. Their meanings are so close that they are mixed sometimes, in both everyday discourse and academic research. This indicates that it is necessary to study them in a unified framework. However, there has been no logical research on them together. In this paper, we propose a language of a bimodal logic with these two concepts, investigate its model-theoretical properties such as expressivity and frame definability. We axiomatize this logic over various classes of frames, whose completeness proofs are shown with the help of a crucial schema. The interactions between contingency and accident can sharpen our understanding of both notions. Then we extend the logic to a dynamic case: public announcements. By finding the required reduction axioms, we obtain a complete axiomatization, which gives us a good application to Moore sentences.
\end{abstract}


\section{Introduction}

Recent years have witnessed a bunch of investigations on non-normal modalities, such as contingency/non-contingency, essence/accident. To say a formula is contingent, if it is possibly true and also possibly false; to say a formula is accidental, if it is true but possibly false. Contingency applies to propositions which have no exact truth value; for example, ``P=NP'', which is possibly true and possibly false. In contrast, accident applies to propositions that are true but possibly false; for example, ``John won the prize'' or ``He is in China''. Despite being definable with other modalities such as necessity, these two modalities formalize various important metaphysical and epistemological notions in their own rights.

The notion of contingency dates back to Aristotle, who \weg{differentiates three kinds of propositions: necessary, impossible, contingent, and}develops a logic of statements about contingency~\cite{Borgan67}. The logical research about this notion is initiated by Montgomery and Routley~\cite{MR66}, followed by Cresswell~\cite{Cresswell88}, Humberstone~\cite{Humberstone95}, Kuhn~\cite{DBLP:journals/ndjfl/Kuhn95}, Zolin~\cite{DBLP:journals/ndjfl/Zolin99}, and Fan, Wang and van Ditmarsch~\cite{Fanetal:2014}. This notion has many analogues in various setting; for example, it corresponds to borderline in a sorites setting, to undecidability in a proof-theoretic setting, to moral indifference in a deontic setting, to agnosticism in a doxastic setting, and to ignorance in an epistemic setting, etc..\footnote{In a recent paper~\cite{Fine:2017}, Fine shows that, in the context of S4 or KD4, knowledge of second order ignorance is impossible, in which `ignorance' means `ignorance whether'.} This means that the technical results on contingency also apply to those analogues. As for a recent study of contingency, we refer to~\cite{Fanetal:2015}.



As a variation of contingency, the notion of accident, or `accidental truths', goes back at least to Leibniz, in disguise of the term `v\'erit\'es de fait' (factual truths) (cf. e.g.~\cite{Heinemann:1948,AD:1989A}).\weg{\footnote{In {\em Monadology}, Leibniz wrote: ``The truths of reason are necessary and their opposite is impossible, the truths of fact are contingent and their opposite is possible.''(e.g.~\cite{Heinemann:1948,AD:1989A}) The notion of `truths of fact' seems quite suitable to be thought of as accidental truths, since it exactly means true and possibly false.}} 
This notion is used to reconstruct G\"odel's ontological argument (e.g.~\cite{Small:2001}), and relevant to the {\em future contingents} problem\weg{in that the latter states that anything true in the future is necessarily true, which has the form $\phi\to\Box\phi$, i.e.} formalized by a negative form of accident~\cite{Aristotle:1941}, and to provide an additional partial verification of the Boxdot Conjecture posed in \cite{FH:2009} (see~\cite{Steinsvold:2011}).

In an epistemic setting, accident is read `unknown truths', which is an important notion in philosophy and formal epistemology. For example, \weg{an unknown truth, formalized as $\varphi\wedge\neg K\varphi$, is a classical example used to refute the {\em knowability thesis} that every truth is knowable.\footnote{It is also called `verificationist thesis' or `verificationist principle' in e.g.~\cite{jfak.fitch:2004}, or `Verification Principle' in e.g.~\cite{Priest:2009}.}} it is a source of Fitch's `paradox of knowability'\weg{: if every truth is knowable, then every truth is known}~\cite{fitch:1963}. \weg{Therefore, all unknown truths are unknowable. }As another example, it is an important kind of Moore sentences, which is in turn essential to Moore's paradox\weg{ a certain class of statements about truths that are not believed or known, the {\em Moore sentences}, is essential to Moore's paradox, which says that one cannot consistently assert a sentence such as `it is raining but I do not know (believe) it is raining', even if the sentence is true }\cite{Moore:1942,hintikka:1962}. In the terminology of dynamic epistemic logic, such a Moore sentence is unsuccessful and self-refuting~\cite{hvdetal.del:2007,hollidayetal:2010,jfak.book:2011}.

To distinguish `accident' from `contingency', a minimal logic of accident is provided in \cite{Marcos:2005}.\weg{\footnote{In \cite{Marcos:2005}, essence is treated as a {\em modality}, which is quite different from the treatments in the literature (see e.g. \cite{Fine:essence}) as the essential properties, something essentially having a property, or something being the essence of something else.}} This axiomatization is then simplified and its various extensions are proposed in \cite{Steinsvold:2008}, which views the work on the logic of accident as a variation and continuation of the work done on contingency logic. Independently of the literature on the logic of accident, \cite{steinsvold:2008b} provides a topological semantics for {\em a logic of unknown truths} and shows its completeness over the class of $\mathcal{S}4$ models. As for a comprehensive treatment of accident logic, see~\cite{GilbertVenturi:2016}.


\medskip

The meanings of contingency and accident are so close that people mix the two notions from time to time in everyday discourse and academic research. For instance, Leibniz used the term `contingency' to mean what is essentially meant by `accident' (e.g.~\cite{Heinemann:1948,AD:1989A}). For another example, in Chinese, the same character has been used to express both notions. Besides, the relationship between the two notions is not clear from the literature. The interactions between contingency and accident may sharpen our understanding of these two concepts. Thus it is necessary to study them in a unified framework.

Despite so many separate investigations on the notions of contingency and of accident in the literature, there has been no logical research on them together. As one can imagine, once we study the two notions at once, the situation may become quite involved. For instance, one difficulty in axiomatizing the logic with contingency and accident as sole primitive modalities, is that we have only one accessibility relation to handle two modal operators uniformly, which makes it nontrivial to find desired interactive axioms of the two notions.



Beyond axiomatizing the logic of contingency and accident over various classes of frames, we also consider the dynamic extension, where contingency and accident operators are better understood as their epistemic counterparts, i.e. `ignorance (or equivalently, not knowing whether)' and `unknown truth', respectively. By applying the usual reduction method, we obtain a complete axiomatization for the dynamic extension of contingency and accident logic.

Our contributions consist of the following:
\begin{enumerate}
\item A schema {\bf NAD} saying that necessity is almost definable in terms of $\Delta$ and $\circ$ (Sec.~\ref{sec.lang-sema})
\item The logic $\mathcal{L}(\nabla,\bullet)$ of contingency and accident is less expressive than standard modal logic over non-reflexive model classes, but equally expressive over reflexive model classes (Sec.~\ref{sec.expre})
\item Transitivity is definable in $\mathcal{L}(\nabla,\bullet)$ with a complex formula (Sec.~\ref{sec.frame-defin})
\item Seriality, reflexivity, Euclideanity, convergency are all undefinable in $\mathcal{L}(\nabla,\bullet)$ by means of a notion of `mirror reduction' (Sec.~\ref{sec.frame-defin})
\item A minimal axiomatization of $\mathcal{L}(\nabla,\bullet)$, which also axiomatizes the class of serial frames (Sec.~\ref{sec.minimal-axiomatization}, Sec.~\ref{sec.serial})
\item An axiomatization of $\mathcal{L}(\nabla,\bullet)$ over transitive frames (Sec.~\ref{sec.transitive})
\item An axiomatization of $\mathcal{L}(\nabla,\bullet)$ over reflexive frames (Sec.~\ref{sec.reflexive})
\item A dynamic extension and one of its applications (Sec.~\ref{sec.dynamic})
\end{enumerate}

\section{Syntax and semantics}\label{sec.lang-sema}

Let $\BP$ be a fixed nonempty set of propositional variables. For the sake of presentation, we introduce a large language, which includes not only contingency operator $\nabla$ and accident operator $\bullet$, but also possibility operator $\Diamond$. But our main focus is the language with only $\nabla$ and $\bullet$ as primitive modalities.

\begin{definition} The language $\mathcal{L}(\nabla,\bullet,\Diamond)$ is generated by the following BNF:
$$\phi::=p\in\BP\mid \neg\phi\mid \phi\land\phi\mid \nabla\phi\mid \bullet \phi\mid \Diamond\phi$$
By disregarding the construct $\Diamond\phi$, we obtain the {\em language $\mathcal{L}(\nabla,\bullet)$ of contingency and accident logic}; by further disregarding the construct $\bullet\phi$ (resp. $\nabla\phi$), we obtain the language $\mathcal{L}(\nabla)$ of contingency logic (resp. $\mathcal{L}(\bullet)$ of accidental logic); by disregarding the constructs $\nabla\phi$ and $\bullet\phi$, we obtain the language $\mathcal{L}(\Diamond)$ of standard modal logic.
\end{definition}

Intuitively, $\nabla\phi$ means ``it is contingent that $\phi$'', $\bullet\phi$ means ``it is accident that $\phi$'', and $\Diamond\phi$ means ``it is possible that $\phi$''. Other connectives and operators are defined as usual; in particular, $\Delta\phi$, $\circ\phi$, $\Box\phi$ abbreviate $\neg\nabla\phi$, $\neg\bullet\phi$ and $\neg\Diamond\neg\phi$, respectively, read as ``it is non-contingent that $\phi$'', ``it is essential that $\phi$'', ``it is necessary that $\phi$''.

$\mathcal{L}(\nabla, \bullet,\Diamond)$ is interpreted over Kripke structures. A (Kripke) model for $\mathcal{L}(\nabla,\bullet,\Diamond)$ is a triple $\M=\lr{S,R,V}$, where $S$ is a nonempty set of possible worlds, $R$ is a binary relation over $S$, called `accessibility relation', and $V$ is a valuation map from $\BP$ to $\mathcal{P}(S)$.

\begin{definition}
Given a model $\M=\lr{S,R,V}$ and $s\in S$, the semantics of $\mathcal{L}(\nabla,\bullet,\Diamond)$ is defined inductively in the following.
\[
\begin{array}{|lll|}
\hline
\M,s\vDash p&\iff & s\in V(p)\\
\M,s\vDash\neg\phi&\iff&\M,s\nvDash\phi\\
\M,s\vDash\phi\land\psi&\iff&\M,s\vDash\phi\text{ and }\M,s\vDash\psi\\
\M,s\vDash\nabla\phi&\iff &\text{there are }t,u\in S\text{ such that }sRt,sRu\text{ and }\M,t\vDash\phi,\M,u\nvDash\phi\\
\M,s\vDash\bullet\phi&\iff &\M,s\vDash\phi\text{ and there exists }t\in S\text{ such that }sRt\text{ and }\M,t\nvDash\phi\\
\M,s\vDash\Diamond\phi&\iff &\text{there are }t\in S\text{ such that }sRt\text{ and }t\vDash\phi\\
\hline
\end{array}
\]
\end{definition}

One may easily compute the semantics of the defined modalities as follows:
\[
\begin{array}{lll}
\M,s\vDash\Delta\phi&\iff &\text{for any }t,u\in S\text{ such that }sRt,sRu,\text{ we have }(\M,t\vDash\phi\iff \M,u\vDash\phi)\\
\M,s\vDash\circ\phi&\iff &\text{if }\M,s\vDash\phi\text{ then for any }t\in S\text{ such that }sRt,\text{ we have }\M,t\vDash\phi\\
\M,s\vDash\Box\phi&\iff &\text{for any }t\in S\text{ such that }sRt\text{ we have }\M,t\vDash\phi\\
\end{array}
\]

\begin{fact}\label{fact.definable} The following results are immediate by the semantics:
\begin{enumerate}
\item[(i)] $\vDash\nabla\phi\lra(\Diamond\phi\land\Diamond\neg\phi)$
\item[(ii)] $\vDash\bullet\phi\lra(\phi\land\Diamond\neg\phi)$
\end{enumerate}
\end{fact}

As shown above, $\nabla$ and $\bullet$ are both definable in terms of $\Diamond$, thus $\mathcal{L}(\Diamond)$ is at least as expressive as $\mathcal{L}(\nabla,\bullet)$.

The following two formulas characterize the relationship between notions of contingency and accident. Intuitively, (1) says that if something is contingent, then either it or its negation is accident, (2) says that if it is accident that something implies anything, and it is also accident that its negation implies anything, then it is contingent. In fact, as we will see in Sec.~\ref{sec.minimal-axiomatization}, the two formulas constitute the desired `bridge axioms' in the minimal axiomatization of $\mathcal{L}(\nabla,\bullet)$.
\begin{proposition}\label{prop.bridge-axioms}\
\begin{enumerate}
\item[(1)] $\vDash\nabla\phi\to\bullet\phi\vee\bullet\neg\phi$
\item[(2)] $\vDash\bullet(\phi\to\psi)\land\bullet(\neg\phi\to\chi)\to\nabla\phi$
\end{enumerate}
\end{proposition}

On one hand, we can see the similarity between contingent and accident: if we replace $\bullet$ with $\nabla$, then the resulted formulas are also valid, since we have $\vDash\nabla\phi\lra\nabla\neg\phi$ and $\vDash\nabla(\phi\to\psi)\land\nabla(\neg\phi\to\chi)\to\nabla\phi$.\footnote{For the latter, consider its equivalence $\Delta\phi\to\Delta(\phi\to\psi)\vee\Delta(\neg\phi\to\chi)$.} On the other hand, we can also see the difference between the two notions: if we replace $\nabla$ in (2) with $\bullet$, then the obtained formula $\bullet(\phi\to\psi)\land\bullet(\neg\phi\to\chi)\to\bullet\phi$ is invalid, as one may easily verify, though its weaker version $\bullet(\phi\to\psi)\land\bullet(\neg\phi\to\chi)\to\bullet\phi\vee\bullet\neg\phi$ is indeed valid.

\medskip




By way of concluding this section, we propose a crucial schema. Recall that a so-called `almost definability' schema AD is proposed in~\cite{Fanetal:2014,Fanetal:2015}, i.e. $\nabla\psi\to(\Box\phi\lra\Delta\phi\land\Delta(\psi\to\phi))$, stating that necessity is almost definable in terms of $\Delta$, which helps find the desired canonical relation in the completeness proof in the cited papers. Since now we need also deal with the clause $\bullet\phi$, the schema AD is not enough. We thus need a new schema that combines $\nabla$ and $\bullet$, if any. Fortunately, we find out the following desired schema, dubbed `{\bf NAD}', which stands for ``{\bf N}ecessity is {\bf A}lmost {\bf D}efinable in terms of $\Delta$ and $\circ$'', to distinguish it from the schema AD. Note that there would appear to be no reason to obtain {\bf NAD} from AD.
$$\bullet\psi\to (\Box\phi\lra \Delta \phi\land\circ(\neg\psi\to\phi))~~~~~~{\bf (NAD)}$$
\begin{proposition}
{\bf (NAD)} is a validity in $\mathcal{L}(\nabla,\bullet)$.
\end{proposition}

\begin{proof}
Let $\M=\lr{S,R,V}$ be a model and $s\in S$. Suppose that $\M,s\vDash\bullet\psi$, to show $\M,s\vDash\Box\phi\lra\Delta\phi\land\circ(\neg\psi\to\phi)$. It should be clear that $\M,s\vDash\Box\phi\to\Delta\phi\land\circ(\neg\psi\to\phi)$. It suffices to show that $\M,s\vDash\Delta\phi\land\circ(\neg\psi\to\phi)\to\Box\phi$.

For this, assume that $\M,s\vDash\Delta\phi\land\circ(\neg\psi\to\phi)$. By supposition, we have $\M,s\vDash\psi$ and $\M,t\vDash\neg\psi$ for some $t$ with $sRt$. Then $\M,s\vDash\neg\psi\to\phi$, which combining with $\M,s\vDash\circ(\neg\psi\to\phi)$ and $sRt$ gives us $\M,t\vDash\neg\psi\to\phi$. Thus $\M,t\vDash\phi$. Since $s\vDash\Delta\phi$, it follows that for all $u$ such that $sRt$, we have $u\vDash\phi$, namely $\M,s\vDash\Box\phi$.
\end{proof}

This schema will guide us to define a suitable canonical relation in the completeness proofs below.

\section{Expressivity results}\label{sec.expre}

$\mathcal{L}(\nabla,\bullet)$ is more expressive than both $\mathcal{L}(\nabla)$ and $\mathcal{L}(\bullet)$ on the class of $\mathcal{K}$-models, $\mathcal{B}$-models, $4$-models, $5$-models (since $\mathcal{L}(\nabla)$ and $\mathcal{L}(\bullet)$ are incomparable on these model classes~\cite[Sec.~3.2]{Fan:2017}), but equally expressive as both logics on the class of $\mathcal{T}$-models (since $\mathcal{L}(\nabla)$ and $\mathcal{L}(\bullet)$ are equally expressive on the model class\cite[Sec.~3.3]{Fan:2017}). In the sequel, we compare the expressive powers of $\mathcal{L}(\nabla,\bullet)$ and $\mathcal{L}(\Diamond)$. As shown in Fact~\ref{fact.definable}, $\mathcal{L}(\Diamond)$ is at least as expressive as $\mathcal{L}(\nabla,\bullet)$ on any class of models.

\begin{proposition}\label{prop.lessexp-k}
$\mathcal{L}(\nabla,\bullet)$ is less expressive than $\mathcal{L}(\Diamond)$ on the class of $\mathcal{K}$-models, $\mathcal{B}$-models, $4$-models, $5$-models, but equally expressive as $\mathcal{L}(\Diamond)$ on the class of $\mathcal{T}$-models.
\end{proposition}

\begin{proof}
\weg{First, note that there is a truth-preserving translation $f:\mathcal{L}(\nabla,\bullet)\to\mathcal{L}(\Diamond)$ defined in the following way:
\[
\begin{array}{lll}
f(p)&=&p\\
f(\neg\phi)&=&\neg f(\phi)\\
f(\phi\land\psi)&=&f(\phi)\land f(\psi)\\
f(\nabla\phi)&=&\Diamond f(\phi)\land \Diamond \neg f(\phi)\\
f(\bullet\phi)&=& f(\phi)\land \Diamond \neg f(\phi)\\
\end{array}
\]
Therefore, $\mathcal{L}(\nabla,\bullet)\preceq \mathcal{L}(\Diamond)$.}

As for the strictness part, consider the following $\mathcal{K}$- (and also $\mathcal{B}$-, $4$-, $5$-) models:

\smallskip

\[
\xymatrix{\M&s:p\ar@(ul,ur)&&&\M'&s':p}
\]

It is straightforward to prove that $\mathcal{L}(\nabla,\bullet)$ formulas cannot distinguish $(\M,s)$ and $(\M',s')$, but $\mathcal{L}(\Diamond)$ can, since $\M,s\vDash\Diamond \top$ whereas $\M',s'\nvDash\Diamond \top$.
\end{proof}

\begin{proposition}
$\mathcal{L}(\nabla,\bullet)$ is less expressive than $\mathcal{L}(\Diamond)$ on the class of $\mathcal{D}$-models.
\end{proposition}

\begin{proof}

Consider the following pointed models $(\M,s)$ and $(\N,s')$, which can be distinguished by an $\ML$-formula $\Box\Box p$\weg{, but cannot be distinguished by any $\LEA$-formulas}:
\medskip
$$
\xymatrix{\mathcal{M}:\ \ \ s:p\ar[rr]&&t:\neg p\ar@(ur,ul)\ar[ll] & & \mathcal{N}:\ \ \ s':p\ar[rr]&& t':\neg p\ar[ll]}
$$

Note that $\M$ and $\N$ are both serial. However, $(\M,s)$ and $(\N,s')$ cannot be distinguished by any $\mathcal{L}(\nabla,\bullet)$-formulas. To show this, we proceed with induction on $\phi\in\mathcal{L}(\nabla,\bullet)$. The nontrivial cases consist of $\nabla\phi$ and $\bullet\phi$. For the case $\nabla\phi$, note that both $s$ and $s'$ have a sole successor, which implies $\M,s\nvDash\nabla\phi$ and $\N,s'\nvDash\nabla\phi$, and thus $\M,s\vDash\nabla\phi$ iff $\N,s'\vDash\nabla\phi$.

For the case $\bullet\phi$, we show by simultaneous induction a stronger result: for all $\phi$, (i) $\M,s\vDash\phi$ iff $\N,s'\vDash\phi$, and (ii) $\M,t\vDash\phi$ iff $\N,t'\vDash\phi$.

For (i), we have the following equivalences:
\[
\begin{array}{lll}
\M,s\vDash\bullet\phi&\Longleftrightarrow& s\vDash\phi\text{ and }t\nvDash\phi\\
&\Longleftrightarrow& s'\vDash\phi\text{ and }t\nvDash\phi\\
&\Longleftrightarrow&s'\vDash\phi\text{ and }t'\nvDash\phi\\
&\Longleftrightarrow& \N,s'\vDash\bullet\phi,\\
\end{array}\]
where the second equivalence followed from the induction hypothesis for (i), and the third equivalence is obtained by (ii).
\[
\begin{array}{lll}
\M,t\vDash\bullet\phi&\Longleftrightarrow& t\vDash\phi\text{ and }(s\nvDash\phi\text{ or }t\nvDash\phi)\\
&\Longleftrightarrow& t\vDash\phi\text{ and }s\nvDash\phi\\
&\Longleftrightarrow& t'\vDash\phi\text{ and }s\nvDash\phi\\
&\Longleftrightarrow&t'\vDash\phi\text{ and }s'\nvDash\phi\\
&\Longleftrightarrow& \N,t'\vDash\bullet\phi,\\
\end{array}\]
where the third equivalence followed from the induction hypothesis for (ii), and the fourth equivalence is obtained by (i).

We have thus completed the proof.
\end{proof}

In summary, on the class of $\mathcal{K}$- (and also $\mathcal{D}$-, $\mathcal{B}$-, $4$-, $5$-) models, the expressive power of $\mathcal{L}(\nabla,\bullet)$ is between $\mathcal{L}(\nabla)$ and $\mathcal{L}(\Diamond)$, and also between $\mathcal{L}(\bullet)$ and $\mathcal{L}(\Diamond)$; on the class of $\mathcal{T}$-models, all logics in question are equally expressive.

\section{Frame Definability}\label{sec.frame-defin}

In the previous section we have seen that $\mathcal{L}(\nabla,\bullet)$ is more expressive than both $\mathcal{L}(\nabla)$ and $\mathcal{L}(\bullet)$ (at the level of models), we may expect that the same situation holds at the level of frames. Recall that many frame properties, in particular transitivity, are undefinable in both sublanguages. Below we shall show that the property of transitivity is definable with a complex formula in the combined language, therefore the new logic is indeed more expressive at the level of frames.

Symmetry is definable in $\mathcal{L}(\bullet)$ with $\bullet(p\to\bullet p)\to p$~\cite[Prop.~10]{Fan:2015accident}, thus also definable in the stronger logic $\mathcal{L}(\nabla,\bullet)$.


\begin{proposition}\label{prop.tran-defin}
The property of transitivity is defined by the following formula:
$$(Tr)~~~\bullet q\land\Delta p\land\circ(\neg q\to p)\to\circ(\neg q\to\circ(\neg r\to p)).$$
\end{proposition}

\begin{proof}
Let a frame $\mathcal{F}=\lr{S,R}$ be given.

Suppose, for a contradiction, that $\mathcal{F}$ is transitive but $\mathcal{F}\nvDash Tr$. That is, there is a model $\M$ based on $\mathcal{F}$ and a state $s\in S$ such that $\M,s\vDash \bullet q\land\Delta p\land\circ(\neg q\to p)$ but $s\nvDash\circ(\neg q\to\circ(\neg r\to p))$. It follows from the latter that there exists $t$ such that $sRt$ and $t\vDash\neg q$ and $t\nvDash\circ(\neg r\to p)$, which implies that there is a $u$ such that $tRu$ and $u\vDash\neg r\land\neg p$. By the transitivity of $R$, we have $sRu$. Moreover, since $s\vDash\bullet q$, it follows that $s\vDash q$ and there exists $t'$ such that $sRt'$ and $t'\vDash\neg q$. From $s\vDash q$ we obtain $s\vDash\neg q\to p$, then by $s\vDash\circ(\neg q\to p)$ and $sRt'$, we can show that $t'\vDash \neg q\to p$, and thus $t'\vDash p$. Now there are two successors $t'$ and $u$ of $s$ which have difference truth values for $p$, hence $s\nvDash\Delta p$, which contradicts with the supposition $s\vDash\Delta p$.

Assume that $\mathcal{F}$ is not transitive, i.e., there are $s,t,u\in S$ such that $sRt$, $tRu$, but {\em not} $sRu$. Clearly, $s\neq t$ and $t\neq u$. Define a valuation $V$ on $\mathcal{F}$ as follows:
$$V(p)=V(r)=S\backslash \{u\}, V(q)=\{s\}.$$
We will show $\lr{\mathcal{F},V},s\nvDash Tr$, which implies $\mathcal{F}\nvDash Tr$.
\begin{itemize}
\item $s\vDash\bullet q$: since $s\neq t$ and $V(q)=\{s\}$, thus $t\nvDash q$. We have also $s\vDash q$ and $sRt$, thus $s\vDash\bullet q$.
\item $s\vDash\Delta p$: this is because for all $w$ such that $sRw$, $w\neq u$, thus by the definition of $V(p)$, $w\vDash p$.
\item $s\vDash\circ(\neg q\to p)$: we have shown in the second item that for all $w$ such that $sRw$, $w\vDash p$, thus $w\vDash\neg q\to p$.
\item $s\nvDash\circ(\neg q\to\circ(\neg r\to p))$: since $t\neq u$, by the definition of $V(p)$, we obtain $t\vDash p$, thus $t\vDash\neg r\to p$; moreover, by the definition of $V(p)$ and $V(r)$, we infer $u\nvDash \neg r\to p$, thus $t\nvDash\circ(\neg r\to p)$ since $tRu$. We have also $t\vDash \neg q$, thus $t\nvDash\neg q\to\circ(\neg r\to p)$. Furthermore, $s\vDash q$, thus $s\vDash\neg q\to\circ(\neg r\to p)$, then we conclude that $s\nvDash\circ(\neg q\to\circ(\neg r\to p))$ due to $sRt$.
\end{itemize}
\end{proof}

In the remainder of this section, we show that none of the properties of seriality, reflexivity, Euclideanity and convergency is definable in $\mathcal{L}(\nabla,\bullet)$. For this, we introduce a notion of `mirror reduction'\footnote{This notion is different from the notion `mirror reduction' in~\cite{Marcos:2005}.}. Intuitively, the mirror reduction of a frame is obtained by deleting all arrows from each $x$ to its sole successor $x$. It is easy to see that every frame has a sole mirror reduction.

\begin{definition}[Mirror reduction] Let $\mathcal{F}=\lr{S,R}$ be a frame. Frame $\mathcal{F}=\lr{S,R^m}$ is said to be the {\em mirror reduction} of $\mathcal{F}$, if
$$R^m=R\backslash\{(x,x)\mid R(x)=\{x\}\}.\footnote{An alternative definition of $R^m$ is such that $R\backslash\{(x,x)\mid R(x)=\{x\}\}\subseteq R^m\subseteq R$. In this case, every frame may have many mirror reductions.}$$
\end{definition}

\begin{proposition}\label{prop.mirror-reduction}
Let $\mathcal{F}^m=\lr{S,R^m}$ be the mirror reduction of $\mathcal{F}=\lr{S,R}$. Then for all $\phi\in\mathcal{L}(\nabla,\bullet)$, we have $$\mathcal{F}^m\vDash\phi\iff\mathcal{F}\vDash\phi.$$
\end{proposition}

\begin{proof}
We show a stronger result: for all $\phi\in\mathcal{L}(\nabla,\bullet)$, for all $s\in S$ and $V$ on $\mathcal{F}$, we have $\lr{\mathcal{F}^m,V},s\vDash\phi$ iff $\lr{\mathcal{F},V},s\vDash\phi$.

We proceed with induction on $\phi$. Boolean cases are straightforward. The only nontrivial cases are $\nabla\phi$ and $\bullet\phi$. In either case, the `only if' direction is easy since $R^m\subseteq R$.

As for the `if' part, suppose $\lr{\mathcal{F},V},s\vDash\nabla\phi$, then there exists $t,u\in S$ such that $sRt$ and $sRu$ and $t\vDash\phi$ and $u\nvDash\phi$. Obviously, $R(s)\neq \{s\}$. Thus $sR^mt$ and $sR^mu$. By induction hypothesis, we conclude that $\lr{\mathcal{F}^m,V},s\vDash\nabla\phi$.

Suppose $\lr{\mathcal{F},V},s\vDash\bullet\phi$, then $s\vDash\phi$ and there exists $t\in S$ such that $sRt$ and $t\nvDash\phi$. Obviously, $s\neq t$, and thus $R(s)\neq\{s\}$. Thus $sR^mt$. By induction hypothesis, we conclude that $\lr{\mathcal{F}^m,V},s\vDash\bullet\phi$.
\end{proof}

\begin{corollary}
Seriality, reflexivity, Euclideanity and convergency are all not definable in $\mathcal{L}(\nabla,\bullet)$.
\end{corollary}

\begin{proof}
Consider the following frames:
\medskip
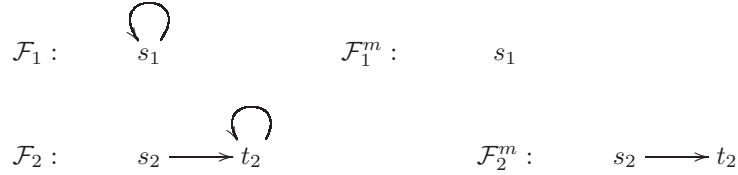
\[
\xymatrix{\mathcal{F}_1:&s_1\ar@(ur,ul)&&\mathcal{F}_1^m:&s_1\\
\mathcal{F}_2:&s_2\ar[r]&t_2\ar@(ur,ul)&&\mathcal{F}_2^m:&s_2\ar[r]&t_2\\}
\]
It is easy to see that $\mathcal{F}_1^m$ and $\mathcal{F}_2^m$ are the mirror reductions of $\mathcal{F}_1$ and $\mathcal{F}_2$, respectively. By Prop.~\ref{prop.mirror-reduction}, we have that for all $\phi\in\mathcal{L}(\nabla,\bullet)$, $\mathcal{F}^m_1\vDash\phi$ iff $\mathcal{F}_1\vDash\phi$, and $\mathcal{F}^m_2\vDash\phi$ iff $\mathcal{F}_2\vDash\phi$. Now observer that $\mathcal{F}_1$ is reflexive and serial, but $\mathcal{F}_1^m$ is not. Thus reflexivity and seriality are not definable in $\mathcal{L}(\nabla,\bullet)$. Moreover, $\mathcal{F}_2$ is Euclidean and convergent, but $\mathcal{F}_2^m$ is not. Thus Euclideanity and convergency are not definable in $\mathcal{L}(\nabla,\bullet)$.
\end{proof}

\paragraph*{Note} Our definition for mirror reduction amounts to a combination of `R-reduction' in~\cite{Humberstone95} and `mirror reduction' in~\cite{Marcos:2005}, since we need to deal with the cases $\bullet\phi$ and $\nabla\phi$ at the same time. It is noteworthy that our definition cannot be replaced by the two notions in the cited papers, which will be explicated as follows.

We recall the `R-reduction' in~\cite{Humberstone95}, where $R^m$ is defined such that
$$R\backslash \{(x,y)\mid R(x)=\{y\}\}\subseteq R^m\subseteq R,$$
i.e. $\mathcal{F}^m$ is obtained from $\mathcal{F}$ by leaving out the arrow from $x$ to its sole successor $y$. This definition cannot give us
Prop.~\ref{prop.mirror-reduction}, for example,

\[
\xymatrix{\mathcal{F}:&s\ar[r]&t&&\mathcal{F}^m&s&t\\
}
\]

It is easy to see that $\mathcal{F}^m$ is a R-reduction of $\mathcal{F}$. However, $\mathcal{F}^m\vDash\circ p$ but $\mathcal{F}\nvDash\circ p$.
 
The notion of `mirror reduction' in~\cite{Marcos:2005} is defined such that
$$R\backslash \{(x,x)\mid x\in S\}\subseteq R^m\subseteq R,$$
i.e. $\mathcal{F}^m$ is obtained from $\mathcal{F}$ by leaving out some or all reflexive arrows. This definition cannot give us
Prop.~\ref{prop.mirror-reduction} either. Take the following frames as an example.

\medskip

\[
\xymatrix{\mathcal{F}:&s\ar[r]\ar@(ul,ur)&t\ar[l]\ar@(ur,ul)&&\mathcal{F}^m&s\ar[r]&t\ar[l]\\
}
\]

Note that $\mathcal{F}^m$ is a mirror reduction in the sense of~\cite{Marcos:2005}. However, it is easy to see that $\mathcal{F}^m\vDash \Delta p$, but $\mathcal{F}\nvDash\Delta p$.

\section{Minimal axiomatization}\label{sec.minimal-axiomatization}

From now on, we axiomatize $\mathcal{L}(\nabla,\bullet)$ over various frame classes. The minimal system is described in the following definition.






\begin{definition}[System $\mathbf{K}^{\nabla\bullet}$] The minimal system of $\mathcal{L}(\nabla,\bullet)$, denoted $\mathbf{K}^{\nabla\bullet}$, includes the following axiom schemas and is closed under the following inference rules.
\[\begin{array}{lllll}
\text{A0}&\text{All instances of tautologies}&&\text{A1}&\bullet\phi\to\phi\\


\text{A2}&\nabla\phi\lra\nabla\neg\phi&&\text{A3}&\bullet(\psi\to\phi)\land\phi\to\bullet\phi\\


\text{A4}&\nabla(\phi\land\psi)\to\nabla\phi\vee\nabla\psi&&\text{A5}&\bullet(\phi\land\psi)\to\bullet\phi\vee\bullet\psi\\

\text{A6}&\nabla\phi\to\bullet\phi\vee\bullet\neg\phi&&\text{A7}&\bullet(\phi\to\psi)\land\bullet(\neg\phi\to\chi)\to\nabla\phi\\

\text{R1}&\dfrac{\phi}{\Delta\phi}&&\text{R2}&\dfrac{\phi}{\circ\phi}\\



\text{R3}&\dfrac{\phi\lra\psi}{\Delta\phi\lra\Delta\psi}&&\text{R4}&\dfrac{\phi\lra\psi}{\circ\phi\lra\circ\psi}\\

\text{MP}&\dfrac{\phi,\phi\to\psi}{\psi}\\
\end{array}\]
\end{definition}

Intuitively, A1 says that accident is true; A2 says that contingency is closed under negation, that is, something is contingent amounts to saying that its negation is contingent; A3 says that something that is true and accidentally implied by anything is itself accidental, of which one equivalence is $\circ\phi\land\phi\to\circ(\psi\to\phi)$; A4 (resp. A5) says that if a conjunction is contingent (resp. accidental), then at least one conjunct is contingent (resp. accidental). In what follows, we will use the more familiar equivalences of A4 and A5, respectively: $\Delta\phi\land\Delta\psi\to\Delta(\phi\land\psi)$ and $\circ\phi\land\circ\psi\to\circ(\phi\land\psi)$. The intuitions of A6 and A7 have been described before Prop.~\ref{prop.bridge-axioms}.


Recall that in the minimal axiomatization of $\mathcal{L}(\nabla)$, the axiom $\Delta\phi\to\Delta(\phi\to\psi)\vee\Delta(\neg\phi\to\chi)$ is indispensable. In contrast, in the minimal axiomatization of the enlarged language $\mathcal{L}(\nabla,\bullet)$, we do not need it any more, though it is provable in the system, due to the completeness to be shown later.

Before introducing the canonical model, we need a bunch of facts and propositions.

\begin{fact}\label{fact.circtodelta}
$\vdash\circ\phi\land\phi\to\Delta\phi$.
\end{fact}

\begin{fact}\label{fact.multi-conjuncts}
For all $k\in\mathbb{N}$, $\vdash\Delta\chi_1\land\cdots\land\Delta\chi_k\to\Delta(\chi_1\land\cdots\land\chi_k)$, and $\vdash\circ\chi_1\land\cdots\land\circ\chi_k\to\circ(\chi_1\land\cdots\land\chi_k)$.
\end{fact}

\begin{proposition}\label{prop.useful-1} For all $n\geq 1$,
$$\vdash \Delta(\bigwedge^n_{k=1}\chi_k\to\phi)\land\bigwedge^n_{k=1}\Delta\chi_k\land\bigwedge^n_{k=1}\circ(\phi\to\chi_k)\to\phi\vee\Delta\phi.$$
\end{proposition}

\begin{proof}
Let $\chi=\bigwedge^n_{k=1}\chi_k$.
We have the following proof sequences:
\[
    \begin{array}{lll}
    (i)&         \neg\phi\to(\phi\to\chi)&\text{A0}\\                                                       (ii)&\circ(\phi\to\chi)\land\neg\phi\to\Delta(\phi\to\chi)&(i),~\text{A0,~Fact~}\ref{fact.circtodelta}\\
    (iii)  &\Delta(\phi\to\chi)\land\Delta(\chi\land\phi)\to\Delta\phi&\text{A0,~A2,~A4,~R3}\\
     (iv)  &\Delta(\chi\to\phi)\land\Delta\chi\to \Delta(\chi\land\phi)&\text{A4,~A0,~R3}\\
     (v)  &\Delta(\chi\to\phi)\land\Delta\chi\land\circ(\phi\to\chi)\to\phi\vee\Delta\phi&(ii)-(iv)\\
     (vi) &\bigwedge^n_{k=1}\Delta\chi_k\to\Delta\chi&\text{Fact~}\ref{fact.multi-conjuncts}\\
     (vii)&\bigwedge^n_{k=1}\circ(\phi\to\chi_k)\to\circ(\phi\to\chi)&\text{Fact~}\ref{fact.multi-conjuncts},\text{~A0,~R4}\\
     (viii)&\Delta(\chi\to\phi)\land\bigwedge^n_{k=1}\Delta\chi_k\land\bigwedge^n_{k=1}\circ(\phi\to\chi_k)\to\phi\vee\Delta\phi&(vi)-(viii)\\
    \end{array}
    \]

\end{proof}

\begin{proposition}\label{prop.useful-2} For all $n\geq 1$,
$$\vdash \Delta(\bigwedge^n_{k=1}\chi_k\to\phi)\land\bigwedge^n_{k=1}\circ(\neg\phi\to\chi_k)\land\phi\to\Delta\phi.$$
\end{proposition}

\begin{proof}
We have the following proof sequences:
\[
\begin{array}{lll}
(i) & \phi\to \bigwedge_{k=1}^n(\neg\phi\to\chi_k)& \text{A0}\\
(ii)&\bigwedge^n_{k=1}\circ(\neg\phi\to\chi_k)\land\phi\to \bigwedge_{k=1}^n\Delta(\neg\phi\to\chi_k)&(i),~\text{A0,~Fact~\ref{fact.circtodelta}}\\
(iii)&\bigwedge_{k=1}^n\Delta(\neg\phi\to\chi_k)\to\Delta(\neg\phi\to\bigwedge_{k=1}^n\chi_k)&\text{A4,~A0,~R3}\\
(iv)&\Delta(\neg\phi\to\bigwedge_{k=1}^n\chi_k)\lra \Delta(\neg\bigwedge_{k=1}^n\chi_k\to\phi)&\text{A0,~R3}\\
(v)&\Delta(\bigwedge^n_{k=1}\chi_k\to\phi)\land\Delta(\neg\bigwedge_{k=1}^n\chi_k\to\phi)\to\Delta\phi&\text{A4,~A0,~R3}\\
(vi)&\Delta(\bigwedge^n_{k=1}\chi_k\to\phi)\land\bigwedge^n_{k=1}\circ(\neg\phi\to\chi_k)\land\phi\to\Delta\phi&(ii)-(v)\\
\end{array}
\]
\end{proof}

\begin{proposition}\label{prop.useful-3} For all $n\geq 1$,
$$\vdash \circ(\bigwedge^n_{k=1}\chi_k\to\phi)\land\bigwedge^n_{k=1}\circ(\neg\phi\to\chi_k)\to\circ\phi.$$
\end{proposition}

\begin{proof}
We have the following proof sequences:
\[
\begin{array}{lll}
(i)&\bigwedge^n_{k=1}\circ(\neg\phi\to\chi_k)\to\circ(\neg\bigwedge^n_{k=1}\chi_k\to\phi)&\text{Fact~\ref{fact.multi-conjuncts},~A0,~R4}\\
(ii)&\circ(\bigwedge^n_{k=1}\chi_k\to\phi)\land\circ(\neg\bigwedge^n_{k=1}\chi_k\to\phi)\to\circ\phi&\text{A5,~A0,~R4}\\
(iii)&\circ(\bigwedge^n_{k=1}\chi_k\to\phi)\land\bigwedge^n_{k=1}\circ(\neg\phi\to\chi_k)\to\circ\phi&(i),~(ii)\\
\end{array}
\]
\end{proof}

We are now in a position to define the desired canonical model. The following definition is inspired by the schema {\bf NAD}.
\begin{definition}[Canonical Model] \label{def.cm}
$\M_c=\lr{S_c,R_c,V_c}$ is the {\em canonical model} of ${\bf K^{\nabla\bullet}}$, if
\begin{itemize}
\item $S_c=\{s\mid s\text{ is a maximal consistent set for }{\bf K^{\nabla\bullet}}\}$,

\item $sR_ct$ iff there exists $\psi$ such that $(a)$ $\bullet\psi\in s$, and $(b)$ for all $\phi$, if $\Delta\phi\land\circ(\neg\psi\to\phi)\in s$, then $\phi\in t$,

\item $V^c(p)=\{s\in S_c\mid p\in s\}$.
\end{itemize}
\end{definition}

\begin{lemma}[Truth Lemma]\label{lem.truthlem-mc}
For all $\phi\in \mathcal{L}(\nabla,\bullet)$, for all $s\in S_c$, we have
$$\M_c,s\vDash\phi\iff \phi\in s.$$
\end{lemma}

\begin{proof}
By induction on $\phi$. The nontrivial cases are $\nabla\phi$ and $\bullet\phi$.
\begin{itemize}
\item Case $\nabla\phi$.

      `$\Longleftarrow$': suppose that $\nabla\phi\in s$, by IH, it suffices to find two successors $t_1,t_2$ in $S_c$ of $s$ such that $\phi\in t_1$ and $\neg\phi\in t_2$. By supposition and axiom A6, $\bullet\phi\vee\bullet\neg\phi\in s$, then $\bullet\phi\in s$ or $\bullet\neg\phi\in s$. We consider only the first case, i.e. $\bullet\phi\in s$ (thus $\phi\in s$ by A1), since the second case is similar.
         In the first case, we show that $\{\chi\mid \Delta\chi\land\circ(\neg\phi\to\chi)\in s\}\cup\{\phi\}$ and $\{\chi\mid \Delta\chi\land\circ(\neg\phi\to\chi)\in s\}\cup\{\neg\phi\}$ are both consistent. We denote the two sets $\Gamma_1,\Gamma_2$, respectively.

         If $\Gamma_1$ is inconsistent, then there are $\chi_1,\cdots,\chi_n$ such that $\vdash\chi_1\land\cdots\land\chi_n\to\neg\phi$, and $\Delta\chi_k\land\circ(\neg\phi\to\chi_k)\in s$ for all $1\leq k\leq n$. By R1, we obtain $\Delta(\chi_1\land\cdots\land\chi_n\to\neg\phi)\in s$. Since we have also $\phi\in s$, Prop.~\ref{prop.useful-1} implies that $\Delta\phi\in s$, contradicting the supposition. Similarly, applying Prop.~\ref{prop.useful-2} we can show that $\Gamma_2$ is consistent.

         \medskip

     `$\Longrightarrow$': assume, for a contradiction, that $\M_c,s\vDash\nabla\phi$ but $\nabla\phi\notin s$ (namely $\Delta\phi\in s$). By assumption, there exist $t,u\in S^c$ such that $sR_ct$, $sR_cu$ and $\phi\in t$, $\phi\notin u$. From $sR_ct$ it follows that there exists $\psi$ such that $\bullet\psi\in s$ and, for all $\alpha$, if $\Delta\alpha\land\circ(\neg\psi\to\alpha)\in s$, then $\alpha\in t$. Since $\neg\phi\notin t$ and $\Delta\neg\phi\in s$, we obtain that $\circ(\neg\psi\to\neg\phi)\notin s$, i.e. $\bullet(\phi\to\psi)\in s$. Similarly, from $sR_cu$ it follows that there exists $\chi$ such that $\bullet(\neg\phi\to\chi)\in s$. Then by axiom A7, we conclude that $\nabla\phi\in s$: a contradiction.

\item Case $\bullet\phi$.

      `$\Longleftarrow$': suppose that $\bullet\phi\in s$, then by Axiom A1, $\phi\in s$. By IH, we only need to find a $t\in S_c$ with $sR_ct$ and $\neg\phi\in t$. For this, it suffice to show that $\{\chi\mid \Delta\chi\land\circ(\neg\phi\to\chi)\in s\}\cup\{\neg\phi\}$ is consistent.

      If not, then there are $\chi_1,\cdots,\chi_m$ such that $\vdash\chi_1\land\cdots\land\chi_m\to\phi$, and $\Delta\chi_j\land\circ(\neg\phi\to\chi_j)\in s$ for all $1\leq j\leq m$. By R2, $\vdash\circ(\chi_1\land\cdots\land\chi_m\to\phi)$. Then using Prop.~\ref{prop.useful-3}, we infer $\circ\phi\in s$, contrary to the supposition, as desired.

      \medskip

      `$\Longrightarrow$': assume, for a contradiction, that $\M_c,s\vDash\bullet\phi$ but $\bullet\phi\notin s$ (i.e. $\circ\phi\in s$). Then by IH, $\phi\in s$ and there is a $t\in S_c$ such that $sR_ct$ and $\phi\not\in t$. From $sR_ct$ it follows that there exists $\chi$ such that $\bullet\chi\in s$ and, for all $\beta$, if $\Delta\beta\land\circ(\neg\chi\to\beta)\in s$, then $\beta\in t$. Since $\phi\notin t$, then $\Delta\phi\land\circ(\neg\chi\to\phi)\notin s$. However, from $\circ\phi\in s$ and $\phi\in s$, we obtain $\Delta\phi\in s$ by Fact~\ref{fact.circtodelta} and $\circ(\neg\chi\to\phi)\in s$ by axiom A3, which is a contradiction.
\end{itemize}
\end{proof}

Now it is a routine exercise to obtain the following.
\begin{theorem}
${\bf K^{\nabla\bullet}}$ is sound and strongly complete with respect to the class of all frames.
\end{theorem}


\section{Extensions}
\subsection{Serial system}\label{sec.serial}

We will show that ${\bf K^{\nabla\bullet}}$ also axiomatize the class of serial frames. This result cannot follow from the truth lemma directly, since the canonical relation $R_c$ in Def.~\ref{def.cm} is not necessarily serial. This is indeed the case when all formulas of the form $\circ\psi$ are included in the states in $\M_c$, so that these states have no $R_c$-successors. We call such states `dead ends w.r.t. $R_c$'. So we need to transform $\M_c$ into a serial model, whereas the truth value of each formula at each state is preserved. For this, we follow the strategy called `reflexivizing the dead ends' introduced in~\cite[p.~226]{Humberstone95} and~\cite[p.~89]{Fanetal:2015}.
\begin{theorem}
${\bf K^{\nabla\bullet}}$ is sound and strongly complete with respect to the class of serial frames.
\end{theorem}

\begin{proof}
Define $\M_c=\lr{S_c,R_c,V_c}$ as in Def.~\ref{def.cm}, and construct a model $\M_c^D=\lr{S_c,R^D_c,V_c}$ such that $R^D_c=R_c\cup \{(s,s)\mid s\text{ is a dead end w.r.t. } R_c \text{ in }\M_c\}$. Now $\M^D_c$ is serial.

The remainder is to show that the satisfiability of formulas in $\mathcal{L}(\nabla,\bullet)$ is invariant under the tranformation. That is to show, for all $\phi\in\mathcal{L}(\nabla,\bullet)$, for all $s\in S_c$, $\M_c,s\vDash\phi$ iff $\M_c^D,s\vDash\phi$. We need only consider the cases for $\nabla\phi$ and $\bullet\phi$. If $s$ is not a dead end w.r.t. $R_c$, then the claim is obvious; otherwise, we have $\M_c,s\nvDash\nabla\phi$ and $\M_c^D,s\nvDash\nabla\phi$, and also $\M_c,s\nvDash\bullet\phi$ and $\M_c^D,s\nvDash\bullet\phi$.
\end{proof}

\subsection{Transitive system}\label{sec.transitive}

We now consider the proof system of $\mathcal{L}(\nabla,\bullet)$ over transitive frames, which extends ${\bf K^{\nabla\bullet}}$ with the following axiom schemas. We denote the system ${\bf K4^{\nabla\bullet}}$.
\[
\begin{array}{ll}
\text{A4-1}&\Delta\phi\to\Delta\Delta\phi\\

\text{A4-2}&\Delta\phi\to\circ(\psi\to\Delta\phi)\\

\text{A4-3}&\bullet\psi_1\land\Delta \phi\land\circ(\neg\psi_1\to\phi)\to\Delta\circ(\neg\psi_2\to\phi)\\

\text{A4-4}&\bullet\psi_1\land\Delta \phi\land\circ(\neg\psi_1\to\phi)\to\circ(\neg\psi_1\to\circ(\neg\psi_2\to\phi))\\
\end{array}\]

\begin{proposition}
${\bf K4^{\nabla\bullet}}$ is sound with respect to the class of transitive frames.
\end{proposition}

\begin{proof}
By the soundness of ${\bf K^{\nabla\bullet}}$, we only need to show the validity of four extra axiom schemas. Moreover, A4-4 defines transitivity (Prop.~\ref{prop.tran-defin}). By way of illustration, we prove the validity of axiom A4-3. 

Let $\M=\lr{S,R,V}$ be a transitive model and $s\in S$. Assume towards contradiction that $\M,s\vDash\bullet\psi_1\land\Delta \phi\land\circ(\neg\psi_1\to\phi)$ but $\M,s\nvDash\Delta\circ(\neg\psi_2\to\phi)$. Then there are $t,u$ such that $sRt$, $sRu$ and $t\vDash\circ(\neg\psi_2\to\phi)$ and $u\nvDash\circ(\neg\psi_2\to\phi)$, and then there is an $x$ with $uRx$ and $x\vDash\neg\psi_2\land\neg\phi$. Since $s\vDash\bullet\psi_1$, it follows that $s\vDash\psi_1$ and $y\vDash\neg\psi_1$ for some $y$ with $sRy$, thus $s\vDash\neg\psi_1\to\phi$. Combining this and $\M,s\vDash\circ(\neg\psi_1\to\phi)$, we have: all $R$-successors of $s$ satisfy $\neg\psi_1\to\phi$, in particular $y\vDash\neg\psi_1\to\phi$, and hence $y\vDash\phi$. From $sRu$ and $uRx$ and the transitivity of $R$, we have $sRx$. But $x\nvDash\phi$, thus we can obtain $s\nvDash\Delta\phi$, contrary to the assumption.
\end{proof}

\begin{theorem}
${\bf K4^{\nabla\bullet}}$ is strongly complete with respect to the class of transitive frames.
\end{theorem}

\begin{proof}
It suffices to show that $R_c$ is transitive.
Suppose $sR_ct$ and $tR_cu$. Then there exists $\psi_1$ such that

(1)~~~$\bullet\psi_1\in s$ and

(2)~~~for each $\phi$, if $\Delta\phi\land\circ(\neg\psi_1\to\phi)\in s$, then $\phi\in t$.\\
Moreover, there is a $\psi_2$ such that

(3)~~~$\bullet\psi_2\in t$ and

(4)~~~for each $\phi$, if $\Delta\phi\land\circ(\neg\psi_2\to\phi)\in t$, then $\phi\in u$.

Assume towards a contradiction that $sR_cu$ fails. Then by (1), there must exist $\phi'$ such that (5)~~~$\Delta \phi'\land\circ(\neg\psi_1\to\phi')\in s$, but $\phi'\notin u$. The remainder is to show that $\Delta\phi'\land\circ(\neg\psi_2\to\phi')\in t$, since then using (4), we can arrive at a contradiction.

\medskip

Since $\Delta\phi'\in s$, from axiom A4-1, it follows that $\Delta\Delta\phi'\in s$; also, from axiom A4-2, it follows that $\circ(\neg\psi_1\to\Delta\phi')\in s$. Then using (2), we infer $\Delta\phi'\in t$.

Due to (1) and (5), since $\vdash\bullet\psi_1\land\Delta \phi'\land\circ(\neg\psi_1\to\phi')\to\Delta\circ(\neg\psi_2\to\phi')$ (axiom A4-3), we have $\Delta\circ(\neg\psi_2\to\phi')\in s$; since $\vdash\bullet\psi_1\land\Delta \phi'\land\circ(\neg\psi_1\to\phi')\to\circ(\neg\psi_1\to\circ(\neg\psi_2\to\phi'))\in s$ (Axiom A4-4), we get $\circ(\neg\psi_1\to\circ(\neg\psi_2\to\phi'))\in s$. Then using (2) again, we conclude that $\circ(\neg\psi_2\to\phi')\in t$, as desired.
\end{proof}


We conclude this subsection with a proposition and two conjectures.

\begin{proposition}
$\vdash\Delta\phi\to\Delta\Delta\phi\land\Delta\circ\phi\land\circ\Delta\phi\land\circ\circ\phi$.
\end{proposition}


\begin{conjecture}
$\vdash\Delta\phi\to\circ^n\Delta^m\circ^l\Delta^k\phi$ for all $n,m,l,k\in\mathbb{N}$ such that $n+m+l+k\geq 2$.
\end{conjecture}

\begin{conjecture}
$\vdash\Delta\phi\to\Delta^m\heartsuit\phi$ for all $m\in\mathbb{N}$ such that $m\geq 1$, where $\heartsuit$ is any combinations of $\Delta,\nabla,\bullet,\circ,\neg$.
\end{conjecture}


\subsection{Reflexive system}\label{sec.reflexive}

The proof system of $\mathcal{L}(\nabla,\bullet)$ over reflexive frames, denoted ${\bf T^{\nabla\bullet}}$, is an extension of ${\bf K^{\nabla\bullet}}$ with an extra axiom schema AT:
$$\Delta\phi\land\phi\to\circ(\psi\to\phi).$$

Let us start with the soundness of ${\bf T^{\nabla\bullet}}$. By the soundness of ${\bf K^{\nabla\bullet}}$, we need only show the validity of AT.
\begin{proposition}
$\Delta\phi\land\phi\to\circ(\psi\to\phi)$ is valid over the class of reflexive frames.
\end{proposition}

\begin{proof}
Given any reflexive model $\M=\lr{S,R,V}$ and any state $s\in S$, suppose, for a contradiction, that $\M,s\vDash\Delta\phi\land\phi$ but $\M,s\nvDash\circ(\psi\to\phi)$. From the latter, it follows that there exists $t$ such that $sRt$ and $t\vDash\psi\land\neg\phi$. By the reflexivity of $R$, we have $sRs$. Then by the first supposition and $sRt$, we infer that $t\vDash\phi$: a contradiction.
\end{proof}

As observed above, $R_c$ is not necessarily serial, thus is not necessarily reflexive. To fix this problem so as to gain the completeness, we need to use the reflexive closure of $R_c$.

\begin{definition}[The canonical model for ${\bf T^{\nabla\bullet}}$] Model $\M_c^T=\lr{S_c,R_c^T,V_c}$ is the {\em canonical model} of ${\bf T^{\nabla\bullet}}$, if $S^c$ and $V^c$ is as previous, and $R_c^T$ is the reflexive closure of $R_c$; in symbol, $R_c^T=R_c\cup\{(s,s)\mid s\in S^c\}$.
\end{definition}

It is clear that $\M_T$ is reflexive. Moreover, the truth lemma holds for $\M_T$.
\begin{lemma}
For each $\phi\in\mathcal{L}(\nabla,\bullet)$ and for each $s\in S_c$, $$\M_c^T,s\vDash\phi\iff \phi\in s.$$
\end{lemma}

\begin{proof}
By induction on $\phi$. We need only check the cases $\nabla\phi$ and $\bullet\phi$.
\begin{itemize}
\item Case $\nabla\phi$: Suppose, for a contradiction, that $\M_c^T,s\vDash\nabla\phi$ but $\nabla\phi\notin s$ (i.e. $\Delta\phi\in s$). Then by induction hypothesis, there are $t,u\in S_c$ such that $sR_c^Tt$ and $sR_c^Tu$ and $\phi\in t$ and $\phi\notin u$. It is obvious that $t\neq u$. According to the definition of $R_T$, we consider the following cases.
    \begin{itemize}
    \item $s\neq t$ and $s\neq u$. Then $sR_ct$ and $sR_cu$. In this case, the proof goes as the corresponding part in Lemma~\ref{lem.truthlem-mc}, and we can arrive at a contradiction.
    \item $s=t$ or $s=u$. W.l.o.g. we may assume that $s=t$, and thus $s\neq u$, which implies that $\phi\in s$ and $sR_cu$. Then there exists $\psi$ such that $\bullet\psi\in s$, and for every $\chi$, if $\Delta\chi\land\circ(\neg\psi\to\chi)\in s$, then $\chi\in u$. Since $\Delta\phi\in s$ and $\phi\notin u$, we obtain $\circ(\neg\psi\to\phi)\notin s$. However, from $\Delta\phi\in s$ again and $\phi\in s$ and axiom AT, it follows that $\circ(\neg\psi\to\phi)\in s$: a contradiction.
    \end{itemize}
    The other way around is immediate from the corresponding part in Lemma~\ref{lem.truthlem-mc} and $R_c\subseteq R_c^T$.
\item Case $\bullet\phi$. Suppose, for a contradiction, that $\M_T,s\vDash\bullet\phi$ but $\bullet\phi\notin s$. Then by induction hypothesis, $\phi\in s$ and there is a $t\in S_c$ such that $sR_c^Tt$ and $\phi\notin t$. Obviously, $s\neq t$, thus $sR_ct$. Then the proof continues as the corresponding part in Lemma~\ref{lem.truthlem-mc}, and it will lead to a contradiction. The other way around is immediate from the corresponding part in Lemma~\ref{lem.truthlem-mc} and $R_c\subseteq R_c^T$.
\end{itemize}
\end{proof}

It follows immediately that
\begin{theorem}
${\bf T^{\nabla\bullet}}$ is sound and strongly complete with respect to the class of reflexive frames.
\end{theorem}

\section{Adding dynamic operators}\label{sec.dynamic}

This section generalizes the logic of contingency and accident to the simplest case of the dynamic operator: public announcements. We propose a complete axiomatization for the extended logic, and apply the system to analyse the successful and self-refuting formulas. Our results can be easily extended to the most general case of action models. It is noteworthy that the dynamic considerations in the accident logic is missing in the literature.

\subsection{Axiomatization with announcements}
The language of contingency and accident logic with public announcement, denoted $\mathcal{L}(\nabla,\bullet,[\cdot])$, is obtained from $\mathcal{L}(\nabla,\bullet)$ by adding public announcement operators.
$$\phi::=p\mid \neg\phi\mid \phi\land\phi\mid\nabla\phi\mid\bullet\phi\mid [\phi]\phi.$$
Intuitively, $[\psi]\phi$ is read `after each truthfully public announcement of $\psi$, it is the case that $\phi$'.

Semantically, the public announcement of $\psi$ is evaluated via eliminating all states where $\psi$ does not hold.
$$\M,s\vDash[\psi]\phi\iff\M,s\vDash\psi \text{ implies }\M|_\psi,s\vDash\phi.$$
Where $\M|_\psi$ is the model restriction of $\M$ to the $\psi$-states.

The common reduction axioms in propositional logic with public announcements consist of:

\medskip

$
\begin{array}{ll}
\text{AP}&[\psi]p\lra(\psi\to p)\\
\text{AN}&[\psi]\neg\phi\lra(\psi\to\neg[\psi]\phi)\\
\text{AC}&[\psi](\phi\land\chi)\lra([\psi]\phi\land[\psi]\chi)\\
\text{AA}&[\psi][\chi]\phi\lra[\psi\land[\psi]\chi]\phi\\
\end{array}
$

\medskip

For the logic $\mathcal{L}(\nabla,\bullet,[\cdot])$, we observe the following key axiom schemas:

\[\begin{array}{|ll|}
\hline
\text{A}\nabla&[\psi]\nabla\phi\lra(\psi\to\nabla[\psi]\phi\land\nabla[\psi]\neg\phi)\\
\text{A}\bullet&[\psi]\bullet\phi\lra(\psi\to\bullet[\psi]\phi)\\
\hline
\end{array}\]

Collecting all reduction axioms into the system ${\bf K^{\nabla\bullet}}$, we obtain a proof system ${\bf K^{\nabla\bullet[\cdot]}}$. We also use ${\bf K^{\bullet[\cdot]}}$ for the subsystem without the axioms involving $\nabla$, and ${\bf K^{\nabla[\cdot]}}$ for the subsystem without the axioms involving $\bullet$.

One may compute the following:
\[\begin{array}{ll}
&[\psi]\Delta\phi\lra(\psi\to\Delta[\psi]\phi\vee\Delta[\psi]\neg\phi)\\
&[\psi]\circ\phi\lra(\psi\to\circ[\psi]\phi)\\
\end{array}\]

\begin{theorem}
${\bf K^{\nabla\bullet[\cdot]}}$ is sound and complete with respect to the class of all frames.
\end{theorem}

\begin{proof}
The soundness of ${\bf K^{\nabla[\cdot]}}$ is given in~\cite[Prop.~7.4]{Fanetal:2015}. It suffices to show the validity of axiom $\text{A}\bullet$. Let $(\M,s)$ be an arbitrary pointed model, where $\M=\lr{S,R,V}$.
\weg{\[\begin{array}{ll}
&\M,s\vDash[\psi]\bullet\phi\\
\iff&\M,s\vDash\psi\text{ implies }\M|_\psi,s\vDash\bullet\phi\\
\iff&\M,s\vDash\psi\text{ implies } (\M|_\psi,s\vDash\phi\text{ and there exists }t\in\M|_\psi\text{ such that }\\
&~~~~~~~~~~~~~~~~~~~~~~~~sRt\text{ and }\M|_\psi,t\nvDash\phi)\\
\iff&(\M,s\vDash\psi\text{ implies }\M|_\psi,s\vDash\phi) \text{ and }(\M,s\vDash\psi\text{ implies there exists }t\in\M|_\psi\\
&\text{ such that }sRt\text{ and }\M|_\psi,t\nvDash\phi)\\
\iff&\M,s\vDash[\psi]\phi\text{ and }
\end{array}\]}

Firstly, suppose that $\M,s\vDash[\psi]\bullet\phi$ and $\M,s\vDash\psi$, to show $\M,s\vDash\bullet[\psi]\phi$. By supposition, $\M|_\psi,s\vDash\bullet\phi$. This means that $\M|_\psi,s\vDash\phi$ and there exists $t\in\M|_{\psi}$ such that $sRt$ and $\M|_\psi,t\nvDash\phi$. Then $\M,s\vDash[\psi]\phi$, and moreover, there exists $t\in S$ such that $sRt$ such that $\M,t\vDash\psi$ and $\M|_\psi,t\nvDash\phi$, which entails $\M,t\nvDash[\psi]\phi$. Therefore, $\M,s\vDash\bullet[\psi]\phi$.

Conversely, assume that $\M,s\vDash\psi\to\bullet[\psi]\phi$ and $\M,s\vDash\psi$ (i.e. $s\in\M|_\psi$), to show $\M|_\psi,s\vDash\bullet\phi$. By assumption, $\M,s\vDash\bullet[\psi]\phi$. Then $\M,s\vDash[\psi]\phi$ and there is a $t$ with $sRt$ such that $\M,t\nvDash[\psi]\phi$. From $\M,s\vDash\psi$ and $\M,s\vDash[\psi]\phi$, it follows that $\M|_\psi,s\vDash\phi$; from $\M,t\nvDash[\psi]\phi$, it follows that $\M,t\vDash\psi$ (i.e. $t\in\M|_\psi$) and $\M|_\psi,t\nvDash\phi$. We have thus shown that $\M|_\psi,s\vDash\phi$ and there exists $t\in\M|_\psi$ such that $sRt$ and $\M|_\psi,t\nvDash\phi$. Therefore, $\M|_\psi,s\vDash\bullet\phi$.

\medskip

The completeness of ${\bf K^{\nabla\bullet[\cdot]}}$ reduces to that of ${\bf K^{\nabla\bullet}}$, using the usual reduction method. 
\end{proof}

\subsection{Application: Successful and self-refuting formulas}

To say a formula $\phi$ is {\em successful}, if it still holds after being announced, in symbol $\vDash[\phi]\phi$. Otherwise, we say this formula is {\em unsuccessful}. Moreover, to say a formula is {\em self-refuting}, if its negation always holds after being announced, in symbol $\vDash[\phi]\neg\phi$. In this part, we will show, by syntactic methods, that Moore sentences are not only unsuccessful, but self-refuting, whereas their negations are all successful.

It has already been shown that Moore sentences are unsuccessful and self-refuting, but the proof perspectives are always semantics, that is, $\nvDash[\bullet p]\bullet p$ and $\vDash[\bullet p]\neg\bullet p$, see e.g.~\cite{jfak.fitch:2004,hvdetal.del:2007,hollidayetal:2010}. With the reduction axioms in hand, one may give a proof-theoretical perspective, in a relatively easy way.

\begin{proposition}\label{prop.bull1}
$[\bullet p]\neg\bullet p$ is provable in ${\bf K^{\bullet[\cdot]}}$.
\end{proposition}

\begin{proof}
We just need to see the following proof sequences:
\[
\begin{array}{llll}
[\bullet p]\neg\bullet p&\lra&(\bullet p\to\neg[\bullet p]\bullet p)&\text{AN}\\
&\lra&(\bullet p\to\neg(\bullet p\to\bullet[\bullet p]p))&\text{A}\bullet\\
&\lra&(\bullet p\to\neg(\bullet p\to\bullet(\bullet p\to p)))&\text{AP}\\
&\lra&(\bullet p\to \bullet p\land \neg\bullet (\bullet p\to p))&\text{A0}\\
&\lra&(\bullet p\to\neg\bullet (\bullet p\to p))&\text{A0}\\
\end{array}
\]

Moreover, $\bullet p\to\neg\bullet (\bullet p\to p)$ is provable in ${\bf K^{\bullet[\cdot]}}$. This is because by axiom A1, $\vdash\bullet p\to p$, then applying R2 and Def.~$\circ$, we obtain that $\vdash\neg\bullet(\bullet p\to p)$.
\end{proof}

On the other hand, unlike Moore sentences, their negations are all successful formulas.
\begin{proposition}\label{prop.bull2}
$[\neg\bullet p]\neg\bullet p$ is provable in ${\bf K^{\bullet[\cdot]}}$.
\end{proposition}

\begin{proof}
We observe the following proof sequences:
\[
\begin{array}{llll}
[\neg\bullet p]\neg\bullet p&\lra&(\neg\bullet p\to\neg[\neg\bullet p]\bullet p)&\text{AN}\\
&\lra&(\neg\bullet p\to\neg(\neg\bullet p\to\bullet[\neg\bullet p]p))&\text{A}\bullet\\
&\lra&(\neg\bullet p\to\neg(\neg\bullet p\to\bullet(\neg\bullet p\to p)))&\text{AP}\\
&\lra&(\neg\bullet p\to \neg\bullet p\land \neg\bullet (\neg\bullet p\to p))&\text{A0}\\
&\lra&(\neg\bullet p\to\neg\bullet (\neg\bullet p\to p))&\text{A0}\\
&\lra&(\bullet (\neg\bullet p\to p)\to\bullet p)&\text{A0}\\
\end{array}
\]

The right-hand side of the last equivalence is provable in ${\bf K^{\bullet[\cdot]}}$, as follows.
By axiom A3, we have $\vdash \bullet (\neg\bullet p\to p)\land p\to\bullet p$. Moreover, $\vdash \bullet (\neg\bullet p\to p)\land\neg p\to\bullet p$, as $\vdash\bullet (\neg\bullet p\to p)\to (\neg\bullet p\to p)$ (by axiom A1) and $\vdash \neg p\land (\neg\bullet p\to p)\to\bullet p$ (by axiom A0). Therefore, $\vdash \bullet (\neg\bullet p\to p)\to\bullet p$.
\end{proof}

If we define an operator $[?\psi]$, called `announcement whether $\psi$' in~\cite{ditmarschfan:2015}, such that $[?\psi]\phi\lra[\psi]\phi\land[\neg\psi]\phi$, then by Props.~\ref{prop.bull1} and \ref{prop.bull2}, we have that $[?\bullet p]\neg\bullet p$ is provable, which says that whenever announcing whether $p$ is an unknown truth, $p$ is {\em not} an unknown truth any more.

\weg{By Proposition \ref{prop.redu}, we obtain the following general result.
\begin{proposition}\label{prop.redu.ann} For all $n>0$ and $n\in\mathbb{N}$,
$[\bullet^np]\neg\bullet^np$ and $[\neg\bullet^np]\neg\bullet^np$ are both provable in $\ALA$.
\end{proposition}

A fragment characterizing the formulas $\phi$ such that $\vDash[\bullet^n\phi]\neg\bullet^n\phi$, i.e. $\bullet^n\phi$ is self-refuting, ($\vDash[\neg\bullet^n\phi]\neg\bullet^n\phi$, i.e. $\neg\bullet^n\phi$ is successful, resp.).

A fragment characterizing the formulas $\phi$ such that $\vdash[\bullet\phi]\neg\bullet\phi$, i.e. $\bullet\phi$ is self-refuting, ($\vdash[\neg\bullet\phi]\neg\bullet\phi$, i.e. $\neg\bullet\phi$ is successful, resp.).

$$\phi::=p\mid \neg p\mid \bullet\phi$$}

\section{Concluding words}

In this paper, we proposed a logic $\mathcal{L}(\nabla,\bullet)$ of contingency and accident, which combines the notions of contingency and accident together. We compared the relative expressive powers of this logic and other related logics. We proved that the property of transitivity is definable in terms of a complex formula involving both contingency operator and accident operator, while seriality, reflexivity, Euclideanity and convergency are all undefinable in $\mathcal{L}(\nabla,\bullet)$, by introducing a notion of `mirror reduction'; in contrast, the undefinability results cannot be solved using notions of `R-reduction' and `mirror reduction' in the literature. With the help of a schema, we gave complete axiomatizations of $\mathcal{L}(\nabla,\bullet)$ over $\mathcal{K}$-frames, $\mathcal{D}$-frames, $4$-frames, $\mathcal{T}$-frames. We also investigate a dynamic extension of $\mathcal{L}(\nabla,\bullet)$ and present a complete axiomatization for this logic, which can be applied to prove syntactically that Moore sentences are self-refuting and negations of Moore sentences are successful.

There are a lot of work to be continued, such as axiomatizations of $\mathcal{L}(\nabla,\bullet)$ over symmetric frames and over Euclidean frames, the suitable notion of bisimulation for $\mathcal{L}(\nabla,\bullet)$ and corresponding van Benthem characterization theorem.

\bibliographystyle{plain}
\bibliography{biblio2017,biblio2015}

\end{document}